\documentclass[12pt, a4paper]{article}

\usepackage{indentfirst}
\usepackage[reqno]{amsmath}
\usepackage{amssymb}
\usepackage{amsthm}
\usepackage{enumerate}
\usepackage{cases}
\usepackage[center]{titlesec}
\usepackage{fancyhdr}
\usepackage[colorlinks, 
linkcolor=blue,
anchorcolor=blue, 
citecolor=red]{hyperref}
\allowdisplaybreaks[1]
\numberwithin{equation}{section} 
\newtheorem{lem}{Lemma}[section]
\newtheorem{thm}[lem]{Theorem}

\newtheorem{cor}[lem]{Corollary}

\title{Sobolev Inequality on Manifolds With Asymptotically Nonnegative Bakry-\'Emery Ricci Curvature\thanks{Supported by NSFC Grants No. 12171091, No. 11831005  and LMNS, Fudan.}}
\date{}

\begin{document}
	\author{Yuxin Dong, Hezi Lin, Lingen Lu}
	\maketitle
	\pagestyle{plain}

		\noindent\textbf{ABSTRACT} ~~ In this paper, inspired by \cite{Br2,Jo}, we prove a Sobolev inequality on manifolds with density and asymptotically nonnegative Bakry-\'Emery Ricci curvature.
		\\
		\\
		\noindent\textbf{2020 MR Subject Classification} ~~35R45; 53C21

\section{Introduction}
%The isoperimetric problem has a long history in mathematics which has a crucial link with the eigenvalue problem, Sobolev inequality, and so on. For the classical isoperimetric problem, we refer the reader to \cite{Os, cha, Choe} and references therein. In recent years, the study of the isoperimetric problem with density attracts much attention. For the isoperimetric problem with density, the reader can consult \cite{Mor1,Mor2, Cabre}.

In 2019, S. Brendle \cite{Br1} proved a Sobolev inequality for submanifolds in Euclidean space. Moreover, he obtained a sharp isoperimetric inequality for compact minimal submanifolds in Euclidean space with codimension at most 2. In 2020, he generalized the results in \cite{Br1} to the case of manifolds with nonnegative curvature (\cite{Br2}). Recently, we extended \cite{Br2} to manifolds with asymptotically nonnegative curvature (\cite{DLL1}). In 2021, \mbox{F. Johne} \cite{Jo} generalized the results of \cite{Br2} to the case of manifolds with density and nonnegative Bakry-\'Emery Ricci curvature. In this note, we establish a Sobolev inequality in manifolds with density and asymptotically nonnegative Bakry-\'Emery Ricci curvature.

Let $(M,g,w\,d\mathrm{vol}_g)$ be a smooth complete noncompact $n$-dimensional Riemannian manifold with density, where $w$ is a smooth positive density function on $M$ and $d\mathrm{vol}_g$ is the Riemannian volume measure with respect to the metric $g$. As a generalization of Ricci curvature, the Bakry-\'Emery Ricci curvature \cite{BE} of $(M,g,w\,d\mathrm{vol}_g)$ is defined by 
\begin{equation}\label{BE}
	\mathrm{Ric}_\alpha^w=\mathrm{Ric}-D^2(\log w)-
	\frac{1}{\alpha}D\log w\otimes D\log w,
\end{equation}
where $\mathrm{Ric}$ denotes the Ricci curvature of $M$, $D$ is the Levi-Civita connection with respect to the metric $g$ and $\alpha>0$. If the density function $w$ is constant, the Bakry-\'Emery Ricci curvature $\mathrm{Ric}_\alpha^w$ reduces to the Ricci curvature.

Suppose $\lambda:[0,+\infty)\to[0,+\infty)$ is a nonnegative and nonincreasing continuous function satisfying
\begin{equation}\label{b0}
	b_0:=\int_0^{+\infty}s\lambda(s)~ds<+\infty,
\end{equation}
which implies
\begin{equation}\label{b1}
	b_1:=\int_0^{+\infty}\lambda(s)~ds<+\infty.
\end{equation}
A complete noncompact Riemannian manifold $(M,g,w\,d\mathrm{vol}_g)$ with density of dimension $n$ is said to have \textbf{asymptotically nonnegative Bakry-\'Emery Ricci curvature} if there exists a base point $o\in M$ such that
\begin{equation}\label{asy-BE}
	\mathrm{Ric}_\alpha^w(q)\geq-(n+\alpha-1)~\lambda(d(o,q)),\quad \forall q\in M,
\end{equation}
where $d$ is the distance function of $M$. In particular, $\lambda\equiv0$ in \eqref{asy-BE} corresponds to the case treated in \cite{Jo}.

Suppose $h:[0,T)\to\mathbb{R}$ is the unique solution of the initial value problem
\begin{equation}\label{h}
	\left\{
	\begin{aligned}
		&h''(t)=\lambda(t)h(t),\\
		&h(0)=0, h'(0)=1.
	\end{aligned}
	\right.
\end{equation}
By the theory of ordinary differential equations \cite{T}, the solution exists for all time, i.e. $T=\infty$. We remark that $h$ reduces to the radius function, if $\lambda=0$.
Similar to the work of F. Johne \cite{Jo}, we define the \textbf{$\alpha$-asymptotic volume ratio} $\mathcal{V}_\alpha$ of $(M,g,w\,d\mathrm{vol}_g)$ by
\begin{equation}\label{avr-be}
	\mathcal{V}_\alpha:=\lim_{r\to+\infty}\frac{\int_{B_r(o)}w}{(n+\alpha)\int_0^r h^{n+\alpha-1}(t)\ dt},
\end{equation}
where $o$ is the base point and $B_r(o)$ denotes the geodesic ball of radius $r$, i.e. $B_r(o)=\{q\in M:d(o,q)<r\}$. In Theorem \ref{thm-mmd}, we will show a comparison theorem for weighted volumes, to be more precise we will show
$$
\frac{\int_{B_r(o)}w}{(n+\alpha)\int_0^r h^{n+\alpha-1}(t)\ dt}
$$
is a nonincreasing function for $r\in(0,+\infty)$, so $\mathcal{V}_\alpha$ is well defined.

By combining the ABP-method in \cite{Br2, Jo} with some comparison theorems for ordinary differential equations, we establish a Sobolev inequality for a compact domain in manifolds with density, under the asymptotically nonnegative Bakry-\'Emery Ricci curvature as follows.

\begin{thm}\label{thm1.1}
	Let $(M,g,w\,d\mathrm{vol}_g)$ be a smooth complete noncompact $n$-dimensional Riemannian manifold of smooth density $w>0$ and asymptotically nonnegative Bakry-\'Emery Ricci curvature with respect to a base point $o\in M$. Let $\Omega$ be a compact domain in $M$ with boundary $\partial\Omega$, and $f$ be a positive smooth function on $\Omega$. Then
	$$
	\begin{aligned}
		&\int_{\partial\Omega}wf+\int_{\Omega}w|Df|+2b_1(n+\alpha-1)\int_\Omega wf\\
		&\geq(n+\alpha)\mathcal{V}_\alpha^{\frac{1}{n+\alpha}}
		\Big(\frac{1+b_0}{e^{2r_0b_1+b_0}}\Big)^{\frac{n+\alpha-1}{n+\alpha}}
		\Big(\int_\Omega wf^{\frac{n+\alpha}{n+\alpha-1}}\Big)^{\frac{n+\alpha-1}{n+\alpha}},
	\end{aligned}
	$$
	where $r_0=\max\{{d}(o,x)|x\in\Omega\}$, $\mathcal{V}_\alpha$ is the $\alpha$-asymptotic volume ratio of $M$ given by \eqref{avr-be}, $b_0,b_1$ are defined in \eqref{b0} and \eqref{b1}, respectively.
\end{thm}

Taking $f=1$ in Theorem \ref{thm1.1}, we obtain an isoperimetric inequality:

\begin{cor}\label{cor1.2}
	Let $(M,g,w\,d\mathrm{vol}_g)$ be a smooth complete noncompact $n$-dimensional Riemannian manifold of smooth density $w>0$ and asymptotically nonnegative Bakry-\'Emery Ricci curvature with respect to a base point $o\in M$. Then 
	$$
	\int_{\partial\Omega}w\geq\Big(
	(n+\alpha)\mathcal{V}_\alpha^{\frac{1}{n+\alpha}}
	\Big(\frac{1+b_0}{e^{2r_0b_1+b_0}}\Big)^{\frac{n+\alpha-1}{n+\alpha}}-
	2(n+\alpha-1)b_1\big(\int_\Omega w\big)^{\frac{1}{n+\alpha}}\Big)\big(\int_\Omega w\big)^{\frac{n+\alpha-1}{n+\alpha}},
	$$
	where $r_0=\max\{{d}(o,x)|x\in\Omega\}$, $\mathcal{V}_\alpha$ is the $\alpha$-asymptotic volume ratio of $M$ given by \eqref{avr-be}, $b_0,b_1$ are defined in \eqref{b0} and \eqref{b1}.
\end{cor}

	When $w=1,b_0=b_1=\alpha=0$, Corollary \ref{cor1.2} was first given by V. Agostiniani, M. Fogagnolo, and L. Mazziari \cite{AFM} in dimension 3 and obtained by S. Brendle \cite{Br2} for any dimension. In recent years, the study of the isoperimetric problem in manifolds with density attracts much attention, see \cite{Mor2, MM, Cabre, FM}. For more about manifolds with density, we refer the reader to \cite{Mor1,M} and references therein.

\section{Preliminaries}
In this section, we give a proof of the Bishop-Gromov volume comparison theorem for complete noncompact Riemannian manifold with density and asymptotically nonnegative Bakry-\'Emery Ricci curvature.

The following lemma is an almost verbatim combination of Lemma 2.1 and Corollary 2.2 in \cite{PRS}. We should point out, however, that it deals with a different initial value case than \cite{PRS}.

\begin{lem}\label{comBE}
	Let $G$ be a continuous nonnegative function on $[0,+\infty)$ and let $g,\psi$ be solutions to the following problems
	\begin{equation}\label{p1}
	\left\{
	\begin{aligned}
		&g'+g^2\leq G,\quad t\in(0,+\infty),\\
		&g(t)= \frac{\beta}{t}+O(1),\text{ as }t\to0^+,
	\end{aligned}
	\right.
	\quad
	\left\{
	\begin{aligned}
		&\psi''= G\psi,\quad t\in(0,+\infty),\\
		&\psi(0)=0, \psi'(0)=1,
	\end{aligned}
	\right.
	\end{equation}
	where $0<\beta\leq 1$. Then we have
	$$
	g\leq\frac{\psi'}{\psi}\text{ on }(0,+\infty).
	$$
\end{lem}

\begin{proof}
	Observe that the initial conditions imply $\psi\ge0$. Then the Fundamental Theorem of Calculus implies $\psi(t)\ge t$ and $\psi'(t)\ge1$. Let $\phi(t)=t^\beta e^{\int_0^t(g-\frac{\beta}{\tau})\ d\tau}>0$ for $t\in(0,+\infty)$. Similar to the proof of Lemma 2.1 and Corollary 2.2 in \cite{PRS}, it is easy to show that
	\begin{equation}\label{p2}
	\begin{aligned}
		&g=\frac{\phi'}{\phi}, \quad\phi''\leq G\phi,\\
		&\phi(t)=t^\beta(1+O(1)), 
		\text{ as }t\to0^+,\\
		& \lim_{t\to0^+}\frac{\psi(t)}{t}=\frac{\psi'(0)}{1}=1,\\
		&\lim_{t\to0^+}(\phi'\psi-\phi\psi')=
		\lim_{t\to0^+}(g\phi\psi-\phi\psi')=0.
	\end{aligned}
	\end{equation}
	Using \eqref{p1} and \eqref{p2}, we conclude that
	$$
	(\phi'\psi-\psi'\phi)'=\phi''\psi-\phi\psi''\leq {G}(t)\psi\phi-
	{G}(t)\psi\phi=0
	$$
    and
	$$
	\phi'\psi-\psi'\phi\leq 0\text{ on }(0,+\infty).
	$$
	Thus,
	$$
	g=\frac{\phi'}{\phi}\leq\frac{\psi'}{\psi}\text{ on }(0,+\infty).
	$$
\end{proof}

The proof of the following theorem is a close adaption of Theorem A.1 in F. Johne \cite{Jo}.

\begin{thm}\label{thm-mmd}
	Let $(M,g,w\,d\mathrm{vol}_g)$ be a smooth complete noncompact $n$-dimensional Riemannian manifold of smooth density $w>0$ and asymptotically nonnegative Bakery-\'Emery Ricci curvature with respect to a base point $o\in M$. Then the function
	$$
	r\mapsto \frac{\int_{B_r(o)}w}{(n+\alpha)\int_0^r h^{n+\alpha-1}(t)\ dt}
	$$
	is nonincreasing. 
\end{thm}

\begin{proof}
	Fix the base point $o\in M$ and $r>0$, let $D_o=M\setminus \mathrm{cut(o)}$ be the domain of the normal geodesic coordinates centered at $o$. We define $B_r(o)=\{q\in M:d(o,q)<r\}$ and its boundary by $S_r(o)=\partial B_r(o)$. Denote the second fundamental form of the hypersurface $S_r(o)\cap D_o$ by $B$ and  the mean curvature of the geodesic sphere with an inward pointing normal vector by $H$.
	
	Let $\gamma(t):=\exp_o(tv),t\in[0,r]$ be a normal geodesic such that $\gamma(t)\in S_t(o)\cap D_o$. We consider the variation set of hypersurfaces that have a constant distance from $N$. By (1.6) in \cite{Pe}, it is easy to know that
	$$
	\frac{d}{dt}H=-|B|^2-\mathrm{Ric}(\gamma',\gamma'),
	\quad t\in(0,r),
	$$
	provided $\gamma(t)\in S_t(o)\cap D_o$. By the definition of Bakry-\'Emery Ricci curvature \eqref{BE}, we deduce that
%	\begin{equation}\label{p3}
	\begin{align}
		&\nonumber\frac{d}{dt}[H+ \langle D\log w,{\gamma}'\rangle]\\
		&\nonumber=-\frac{1}{\alpha}\langle D\log w,{\gamma}' \rangle^2- \mathrm{Ric}^w_\alpha(\gamma',\gamma')-|B|^2\\
		&\nonumber\leq-\frac{1}{n-1}H^2-\frac{1}{\alpha}\langle D\log w,{\gamma}' \rangle^2
		- \mathrm{Ric}^w_\alpha(\gamma',\gamma')\\
		&\nonumber=-\frac{1}{n+\alpha-1}(H+ \langle D\log w,{\gamma}'\rangle)^2
		-\frac{n-1}{\alpha(n+\alpha-1)}\Big(\frac{\alpha}{n-1}H-
		\langle D\log w,{\gamma}'\rangle\Big)^2
		-\mathrm{Ric}^w_\alpha(\gamma',\gamma')\\
		&\leq-\frac{1}{n+\alpha-1}(H+ \langle D\log w,{\gamma}' \rangle)^2
		-\mathrm{Ric}^w_\alpha(\gamma',\gamma').
			\label{p3}
	\end{align}
%	\end{equation}
	Set $g_o=\frac{1}{n+\alpha-1}[H+ \langle D\log w,{\gamma}'\rangle],t\in(0,r)$. Using $\lim\limits_{t\to 0^+}tH(t)=n-1$ and the smoothness of the density $w$, by \eqref{asy-BE} and \eqref{p3}, we can find
	$$
	\left\{
	\begin{aligned}
		&g_o'+g_o^2\leq \lambda,\quad t\in(0,r),\\
		&g_o(t)=\frac{n-1}{(n+\alpha-1)t}+O(1), 
		\text{ as }t\to0^+.
	\end{aligned}
	\right.
	$$
	Note that  $0<\frac{n-1}{n+\alpha-1}\leq 1$, from \eqref{h} and Lemma \ref{comBE}, it follows
	$$
	g_o\leq\frac{h'}{h},
	$$
	that is,
	$$
	H+ \langle D(\log w),{\gamma}'\rangle\leq(n+\alpha-1)\frac{h'}{h}.
	$$
	By the first variation formula for the manifold with density, we obtain
	$$
	\begin{aligned}
		\frac{d}{dt}\Big(\int_{S_t(o)}w\Big)&=\frac{d}{dt}
		\Big(\int_{S_t(o)\cap D_o}w\Big)
		=\int_{S_t(o)\cap D_o}(H+\langle\nu,D\log w\rangle)w
		\\
		&\leq(n+\alpha-1)\frac{h'}{h}\int_{S_t(o)\cap D_o}w
		=(n+\alpha-1)\frac{h'}{h}\int_{S_t(o)}w, 
	\end{aligned}
	$$
	where $\nu$ is the unit outward vector field along $S_t(o)\cap D_o$. This implies that
	$$
	t\mapsto\frac{\int_{S_t(o)}w}{h^{n+\alpha-1}(t)}
	$$
	is a nonincreasing function. Following Lemma 2.2 in \cite{Zhu}, we derive that
	$$
	\int_{B_r(o)}w=\int_0^r\frac{\int_{S_t(o)}w}{h^{n+\alpha-1}}h^{n+\alpha-1}~dt\ge
	\frac{\int_{S_r(o)}w}{h^{n+\alpha-1}(r)}\int_0^rh^{n+\alpha-1}~dt,
	$$
	which implies
	$$
	\frac{d}{dr}\Big(\frac{\int_{B_r(o)}w}{\int_0^rh^{n+\alpha-1}~dt}\Big)=
	\frac{h^{n+\alpha-1}(r)}{(\int_0^rh^{n+\alpha-1}~dt)}\Big(\frac{\int_{S_r(o)}w}{h^{n+\alpha-1}(r)}-\frac{\int_{B_r(o)}w}{\int_0^rh^{n+\alpha-1}~dt}\Big)\le0.
	$$
	This proves the assertion.
\end{proof}

\section{Proof of Theorem 1.1}
Let $(M,g,w\,d\mathrm{vol}_g)$ be a complete noncompact $n$-dimensional Riemannian manifold with smooth density $w>0$ and asymptotically nonnegative Bakry-\'Emery Ricci curvature with respect to a base point $o\in M$. Let $\Omega$ be a compact and connected domain in $M$ with smooth boundary $\partial\Omega$, and $f$ be a smooth positive function on $\Omega$. 

We only need to prove Theorem \ref{thm1.1} in the case that $\Omega$ is connected. By scaling, we may assume that
\begin{equation}\label{scale}
\int_{\partial\Omega}wf+\int_{\Omega}w|Df|+
2(n+\alpha-1)b_1\int_\Omega wf=(n+\alpha)
\int_{\Omega}wf^{\frac{n+\alpha}{n+\alpha-1}}.
\end{equation}
Due to \eqref{scale} and the connectedness of $\Omega$, we can find a solution to the following Neumann problem
\begin{equation}
	\left\{
	\begin{aligned}
		&\mathrm{div}(wfDu)=(n+\alpha)wf^{\frac{n+\alpha}{n+\alpha-1}}-w|Df|-2(n+\alpha-1)b_1wf& \text{ in }\Omega,\\
		\\
		&\langle Du,\nu\rangle=1,
		& \text{ on }\partial\Omega,
	\end{aligned}
	\right.\label{neumann}
\end{equation}
where $\nu$ is the outward unit normal vector field of $\partial\Omega$. By standard elliptic regularity theory (see Theorem 6.31 in \cite{GT}), we know that $u\in C^{2,\gamma}$ for each $0<\gamma<1$.

Following the notions in \cite{Br2}, we define
$$
U:=\{x\in\Omega\setminus\partial\Omega:|Du(x)|<1\}.
$$
For any $r>0$, we denote $A_r$ by
$$
\{\bar{x}\in U: ru(x)+\frac{1}{2}{d}
(x,\exp_{\bar{x}}(rD u(\bar{x})))^2\geq
ru(\bar{x})+\frac{1}{2}r^2|D u(\bar{x})|^2,
\forall x\in\Omega\}
$$
and the transport map $\Phi_r:\Omega\to M$ by
$$
\Phi_r(x)=\exp_x(rD u(x)), \quad\forall x\in\Omega.
$$
Using the regularity of the solution $u$ of the Neumann problem, we know that the transport map is of class $C^{1,\gamma},0<\gamma<1$.

We obtain the following lemma similar to Lemma 2.1 in \cite{Br2}.

\begin{lem}\label{hessu}
	Assume that $x\in U$. Then we have
	$$
	w\Delta u+\langle Dw,Du\rangle+2(n+\alpha-1)b_1w\leq (n+\alpha)wf^{\frac{1}{n+\alpha-1}}.
	$$
\end{lem}

\begin{proof}
	Using the Cauchy-Schwarz inequality and the property that $|Du|<1$ for $x\in U$, we get
	$$
	-\langle Df,Du\rangle\leq|Df|.
	$$
	In terms of \eqref{neumann}, we derive that
	$$
	\begin{aligned}
		f(w\Delta u+\langle Dw,Du \rangle+2(n+\alpha-1)b_1w)&=(n+\alpha)wf^{\frac{n+\alpha}{n+\alpha-1}}
		-w(|Df|+\langle Df,Du\rangle)\\
		&\leq(n+\alpha)wf^{\frac{n+\alpha}{n+\alpha-1}}.
	\end{aligned}
	$$
	This proves the assertion.
\end{proof}

The proofs of the following three lemmas are identical to those of Lemmas 2.2-2.4 in \cite{Br2} without any change for the case of asymptotically nonnegative Bakry-\'Emery Ricci curvature. So we omit them here.

\begin{lem}[cf. S. Brendle, Lemma 2.2 in \cite{Br2}]\label{trans}
	The set
	\[
	\{q\in M: d(x,q)<r,  \forall x\in\Omega\}
	\]
	is contained in $\Phi_r(A_r)$.
\end{lem}	

\begin{lem}[cf. S. Brendle, Lemma 2.3 in \cite{Br2}]\label{d-mat}
	Assume that $\bar{x}\in A_r$, and let 
	$\bar{\gamma}(t):=\exp_{\bar{x}}(tDu(\bar{x}))$ for all $t\in[0,r]$. 
	If $Z$ is a smooth vector field along $\bar{\gamma}$ satisfying
	$Z(r)=0$, then
	\[
	(D^2u)(Z(0),Z(0))+\int_0^r\big(|D_tZ(t)|^2-R(\bar{\gamma}'(t),Z(t),\bar{\gamma}'(t),Z(t))\big)\ dt\geq0,
	\]	
	where $R$ is the Riemannian curvature tensor of $M$.
\end{lem}

\begin{lem}[cf. S. Brendle, Lemma 2.2 in \cite{Br2}]\label{jacovani}
	Assume that $\bar{x}\in A_r$, and let 
	$\bar{\gamma}(t):=\exp_{\bar{x}}(tD u(\bar{x}))$ for all $t\in[0,r]$. Moreover, 
	let $\{e_1,\dots,e_n\}$ be an orthonormal basis of $T_{\bar{x}}M$. Suppose that $W$ 
	is a Jacobi field along $\bar{\gamma}$ satisfying 	
	\[
	\langle D_tW(0),e_j\rangle=(D^2u)(W(0),e_j),\quad 1\leq j\leq n.
	\]
	If $W(\tau)=0$ for some $\tau\in(0,r)$, then $W$ vanishes identically.
\end{lem}

We now give the proof of Theorem 1.1. The strategy of the proof follows the work of S. Brendle \cite{Br2} closely.

\begin{proof}[\textbf{Proof of Theorem 1.1}]
For any $r>0$ and $\bar{x}\in A_r$, let $\{e_1,\dots,e_n\}$ be an orthonormal basis of the tangent space $T_{\bar{x}}M$. Choose the geodesic normal coordinates $(x^1,\dots,x^n)$ around $\bar{x}$, such that $\frac{\partial}{\partial x^i}=e_i$ at $\bar{x}$. Let $\bar{\gamma}(t):=\exp_{\bar{x}}(tD u(\bar{x}))$ for all $t\in[0,r]$. For $1\leq i\leq n$, let $E_i(t)$ be the parallel transport of $e_i$ along $\bar{\gamma}$. For $1\leq i\leq n$, let $X_i(t)$ be the Jacobi field along $\bar{\gamma}$ with the initial conditions of $X_i(0)=e_i$ and 
$$
\langle D_tX_i(0),e_j\rangle=(D^2u)(e_i,e_j),\quad1\leq j\leq n.
$$
Let $P(t)=(P_{ij}(t))$ be a matrix defined by
$$
P_{ij}(t)=\langle X_i(t),E_j(t)\rangle, \quad
1\leq i,j\leq n.
$$
It follows from Lemma \ref{jacovani} that $X_1(t),\cdots,X_n(t)$ are linearly independent for each $t\in(0,r)$, which implies that the matrix $P(t)$ is invertible for each $t\in(0,r)$. It is obvious that $\det P(t)>0$ if $t$ is sufficiently small. Therefore, $|\det D\Phi_t(\bar{x})|=\det P(t)>0$, for $t\in[0,r)$. Let $S(t)=(S_{ij}(t))$ be a matrix defined by
$$
S_{ij}(t)=R(\bar{\gamma}'(t),E_i(t),\bar{\gamma}'(t),E_j(t)),
\quad
1\leq i,j\leq n,
$$
where $R$ denotes the Riemannian curvature tensor of $M$. By the Jacobi equation, we obtain
\begin{equation}\label{q1}
\left\{
\begin{aligned}
	&P''(t)=-P(t)S(t),\quad t\in[0,r],\\
	\\
	&P_{ij}(0)=\delta_{ij},P_{ij}'(0)=(D^2u)(e_i,e_j).
\end{aligned} \right.
\end{equation}
Let $Q(t)=P(t)^{-1}P'(t),t\in(0,r)$, which is symmetric showed by S. Brendle \cite{Br2}. By \eqref{q1}, a simple computation yields
$$
\frac{d}{dt}Q(t)=-S(t)-Q^2(t).
$$
Recalling that
$$
\mathrm{Ric}_w^{\alpha}:= \mathrm{Ric}-D^2(\log w)
-\frac{1}{\alpha}D\log w\otimes D\log w, 
$$
we follow the computation by F. Johne \cite{Jo} to derive that
$$
\begin{aligned}
	&\frac{d}{dt}[\mathrm{tr} Q+ \langle D\log w,\bar{\gamma}' \rangle]\\
	&=-\frac{1}{\alpha}\langle D\log w,\bar{\gamma}' \rangle^2- \mathrm{Ric}^w_\alpha(\bar{\gamma}',\bar{\gamma}')
	-\mathrm{tr}[Q^2]\\
	&\leq-\frac{1}{n}[\mathrm{tr} Q]^2-\frac{1}{\alpha}\langle D\log w,\bar{\gamma}' \rangle^2
	- \mathrm{Ric}^w_\alpha(\bar{\gamma}',\bar{\gamma}')\\
	&=-\frac{1}{n+\alpha}(\mathrm{tr} Q+ \langle D\log w	,\bar{\gamma}' \rangle)^2
	-\frac{n}{\alpha(n+\alpha)}\Big(\frac{\alpha}{n}\mathrm{tr} Q-	\langle D\log w,\bar{\gamma}' \rangle\Big)^2
	-\mathrm{Ric}^w_\alpha(\bar{\gamma}',\bar{\gamma}')\\
	&\leq-\frac{1}{n+\alpha}(\mathrm{tr} Q+ \langle D\log w,\bar{\gamma}' \rangle)^2
	-\mathrm{Ric}^w_\alpha(\bar{\gamma}',\bar{\gamma}').
\end{aligned}
$$
Set $g=\frac{1}{n+\alpha}[\mathrm{tr} Q+ \langle D(\log w),\bar{\gamma}'(t) \rangle]$. The assumption of asymptotic nonnegative Bakry-\'Emery Ricci curvature gives
\begin{equation}\label{q2}
	g'+g^2\leq\frac{n+\alpha-1}{n+\alpha}|D u(\bar{x})|^2\lambda(d(o,\bar{\gamma}(t))),
\end{equation}
where $o$ is the base point. By the triangle inequality, we get
\begin{equation}\label{q3}
d(o,\bar{\gamma}(t))\geq|d(o,\bar{x})-d(\bar{x},\bar{\gamma}(t))|
=|d(o,\bar{x})-t|D u(\bar{x})||.
\end{equation}
Set
\begin{equation*}
\Lambda_{\bar{x}}(t)=\frac{n+\alpha-1}{n+\alpha}|D u(\bar{x})|^2
\lambda(|d(o,\bar{x})-t|D u(\bar{x})||).
\end{equation*}
Since $\lambda$ is a nonincreasing function, it follows from \eqref{q1}, \eqref{q2} and \eqref{q3} that
$$
\left\{
\begin{aligned}
	&g'(t)+g(t)^2\leq \Lambda_{\bar{x}}(t), \quad t\in(0,r),\\
	\\
	&g(0)=\frac{1}{n+\alpha}[\Delta u(\bar{x})+\langle D(\log w)(\bar{x}),Du(\bar{x})\rangle].
\end{aligned} \right.
$$
Let $\phi=e^{\int_0^tg(\tau)d\tau}$, then 
\begin{equation}\label{q4}
\left\{
\begin{aligned}
	&\phi''\leq \Lambda_{\bar{x}}(t)\phi, \quad t\in(0,r),\\
	\\
	&\phi(0)=1,\phi'(0)=g(0).
\end{aligned} \right.
\end{equation}
Set $\psi_1,\psi_2$ be solutions of the following problems
\begin{equation}\label{q5}
\left\{
\begin{aligned}
	&\psi_1''= \Lambda_{\bar{x}}(t)\psi_1 ,\quad t\in(0,r),\\
	\\
	&\psi_1(0)=0,\psi_1'(0)=1,
\end{aligned} \right.
\quad
\left\{
\begin{aligned}
	&\psi_2''= \Lambda_{\bar{x}}(t)\psi_2 ,\quad t\in(0,r),\\
	\\
	&\psi_2(0)=1,\psi_2'(0)=0.
\end{aligned} \right.
\end{equation}
By the assumption of \eqref{b0}, one knows that $\int_0^{\infty}\Lambda_{\bar{x}}(t)~dt<\infty$. Similar to the proof of Lemma 2.6 in \cite{DLL1}, we have
\begin{equation}\label{psi}
\frac{\psi_2}{\psi_1}(r)\leq \int_0^{+\infty}\Lambda_{\bar{x}}(t)\ dt+\frac{1}{r}
\leq2\frac{n+\alpha-1}{n+\alpha}b_1|Du(\bar{x})|+\frac{1}{r}.
\end{equation}
Noting that $|Du(\bar{x})|<1$, then
\begin{equation}\label{q6}
	\frac{\psi_2}{\psi_1}(r)\leq2\frac{n+\alpha-1}{n+\alpha}b_1+\frac{1}{r}.
\end{equation}
By Lemma 2.13 in \cite{PRS} and \eqref{q5}, we derive that
\begin{equation}\label{comall}
	\begin{aligned}
		\psi_1(t)&\leq 
		\int_0^te^{\int_0^s \tau\Lambda_{\bar{x}}(\tau)~d\tau}~ds
		\leq te^{\int_0^\infty \tau\Lambda_{\bar{x}}(\tau)~d\tau}
		\\
		&=te^{\frac{n+\alpha-1}{n+\alpha}\int_0^\infty v\lambda(|d(o,\bar{x})-v|)~dv}
		\leq te^{\frac{n+\alpha-1}{n+\alpha}(2r_0b_1+b_0)},
	\end{aligned}
\end{equation}
where $r_0=\max\{d(o,x)|x\in\Omega\}$. 

Letting $\psi(t)=\psi_2(t)+g(0)\psi_1(t)$, using \eqref{q4}, \eqref{q5} and Lemma 2.5 in \cite{DLL1}, we obtain
$$
\frac{1}{n+\alpha}[\mathrm{tr}Q+\langle D\log w,\bar{\gamma}'\rangle]=\frac{\phi'}{\phi}\leq\frac{\psi'}{\psi},\quad\forall t\in(0,r).
$$
Consequently,
\begin{equation}\label{q7}
	\frac{d}{dt}\log[w(\bar{\gamma}(t))\det P(t)]=\mathrm{tr}Q(t)+\langle D\log w(\bar{\gamma}(t)),\bar{\gamma}'(t)\rangle\leq (n+\alpha)\frac{\psi'}{\psi}.
\end{equation}
Through \eqref{q7}, we can get
$$
\begin{aligned}
	&w(\Phi_t(\bar{x}))|\det D\Phi_t(\bar{x})|=w(\Phi_t(\bar{x}))\det P(t)\\
	&\leq w(\bar{x})\Big(\psi_2(t)+\frac{1}{n+\alpha}[\Delta u(\bar{x})+\langle D\log w(\bar{x}),Du(\bar{x})\rangle]\psi_1(t)\Big)^{n+\alpha}
\end{aligned}
$$
for all $t\in[0,r]$. This implies
$$
w(\Phi_r(\bar{x}))|\det D\Phi_r(\bar{x})|\leq w(\bar{x})\Big(\frac{\psi_2(r)}{\psi_1(r)}+g(0)\Big)^{n+\alpha}\psi^{n+\alpha}_1(r)
$$
for any $\bar{x}\in A_r$. Using \eqref{q6}, \eqref{comall} and Lemma \ref{hessu}, it follows that
\begin{equation}\label{q8}
\begin{aligned}
	&w(\Phi_r(\bar{x}))|\det D\Phi_r(\bar{x})|
	\\
	&\leq w(\bar{x})\Big(\frac{n+\alpha-1}{n+\alpha}2b_1+\frac{1}{r}+
	\frac{1}{n+\alpha}[\Delta u(\bar{x})+\langle D(\log w)(\bar{x}),Du(\bar{x})\rangle]\Big)^{n+\alpha}\\&\quad\cdot
	r^{n+\alpha}e^{(n+\alpha-1)(2r_0b_1+b_0)}\\
	&\leq w(\bar{x})\Big(\frac{1}{r}+f^\frac{1}{n+\alpha-1}(\bar{x})\Big)^{n+\alpha}
	r^{n+\alpha}e^{(n+\alpha-1)(2r_0b_1+b_0)}
\end{aligned}
\end{equation}
for any $\bar{x}\in A_r$. Moreover, by \eqref{h}, we obtain $h(t)\geq t$ and
\begin{equation}\label{h1234}
	\lim_{t\to\infty}h'(t)=1+\int_0^\infty h(s)\lambda(s)~ds\geq1+\int_0^\infty s\lambda(s)~ds
	=1+b_0.
\end{equation}
Combining Lemma \ref{trans}, \eqref{q8} with the formula for change of variables in multiple integrals, we conclude that
\begin{equation}\label{jfbh} 
\begin{aligned}
	&\int_{\{q\in M:d(x,q)<r\text{ for all }x\in\Omega \}}w ~d\mathrm{vol}_g(x)\\
	&\leq\int_{A_r}|\det D\Phi_r(\bar{x})|w(\Phi_r(\bar{x})) ~d\mathrm{vol}_g(\bar{x})\\
	&\leq\int_{A_r}\Big(\frac{1}{r}+f^\frac{1}{n+\alpha-1}(\bar{x})\Big)^{n+\alpha}
	r^{n+\alpha}e^{(n+\alpha-1)(2r_0b_1+b_0)}w(\bar{x}) ~d\mathrm{vol}_g(\bar{x}).
\end{aligned}
\end{equation}
Let $r>r_0$, the triangle inequality implies that
\begin{equation}\label{inclusion}
	B_{r-r_0}(o)\subset\{q\in M:d(x,q)<r\text{ for all }x\in\Omega \}
	\subset B_{r+r_0}(o).
\end{equation}
Using $\lim\limits_{r\to\infty}\frac{\int_0^{r-r_0}h(t)^{n+\alpha-1}~dt}{\int_0^{r}h(t)^{n+\alpha-1}~dt}=\lim\limits_{r\to\infty}\frac{h(r-r_0)^{n+\alpha-1}}{h(r)^{n+\alpha-1}}$ by the L'Hospital's rule, and combining \eqref{avr-be}, \eqref{inclusion} with Lemma 2.7 in \cite{DLL1}, we have
	\begin{align*}
		\mathcal{V}_\alpha&=\mathcal{V}_\alpha\lim_{r\to\infty}
		\frac{h(r-r_0)^{n+\alpha-1}}{h(r)^{n+\alpha-1}}
		\\
		&=\lim_{r\to\infty}
		\frac{\int_{B_{r-r_0}(o)}w~d\mathrm{vol}_g}{(n+\alpha)\int_0^{r-r_0}h(t)^{n+\alpha-1}~dt}
		\frac{\int_0^{r-r_0}h(t)^{n+\alpha-1}~dt}{\int_0^{r}h(t)^{n+\alpha-1}~dt}
		\\
		&\leq\lim_{r\to\infty}\frac{\int_{\{q\in M:d(x,q)<r\text{ for all }x\in\Omega \}}w~d\mathrm{vol}_g}
		{(n+\alpha)\int_0^rh(t)^{n+\alpha-1}~dt}\\
		&\leq	\lim_{r\to\infty}\frac{\int_{B_{r+r_0}(o)}w~d\mathrm{vol}_g}{(n+\alpha)\int_0^{r+r_0}h(t)^{n-1}~dt}
		\frac{\int_0^{r+r_0}h(t)^{n+\alpha-1}~dt}{\int_0^{r}h(t)^{n+\alpha-1}~dt}	\\
		&=\mathcal{V}_\alpha\lim_{r\to\infty}
		\frac{h(r+r_0)^{n+\alpha-1}}{h(r)^{n+\alpha-1}}\\
		&=\mathcal{V}_\alpha,
	\end{align*}
which implies that
\begin{equation}\label{avr-v}
	\mathcal{V}_\alpha=\lim_{r\to\infty}\frac{\int_{\{q\in M:d(x,q)<r\text{ for all }x\in\Omega \}}w~d\mathrm{vol}_g}
	{(n+\alpha)\int_0^rh(t)^{n+\alpha-1}~dt}.
\end{equation}
Dividing both side of \eqref{jfbh} by $(n+\alpha)\int_0^rh(t)^{n+\alpha-1}dt$ and letting $r\to\infty$, using \eqref{h1234} and \eqref{avr-v}, one can find that
$$
\begin{aligned}
	\mathcal{V}_\alpha&\leq e^{(n+\alpha-1)(2r_0b_1+b_0)}\int_\Omega wf^{\frac{n+\alpha}{n+\alpha-1}}
	\lim_{r\to\infty}\frac{r^{n+\alpha}}{(n+\alpha)\int_0^rh(t)^{n+\alpha-1}~dt}
	\\
	&=e^{(n+\alpha-1)(2r_0b_1+b_0)}\int_\Omega wf^{\frac{n+\alpha}{n+\alpha-1}}\lim_{r\to\infty}\frac{1}{h'(t)^{n+\alpha-1}}
	\\
	&\leq\Big(\frac{e^{2r_0b_1+b_0}}{1+b_0}\Big)^{n+\alpha-1}\int_\Omega
	wf^{\frac{n+\alpha}{n+\alpha-1}}.
\end{aligned}
$$
Under our scaling assumption \eqref{scale}, we obtain
$$
\begin{aligned}
	&\int_{\partial\Omega}wf+\int_{\Omega}w|Df|+(n+\alpha-1)2b_1\int_\Omega wf=(n+\alpha)
	\int_{\Omega}wf^{\frac{n+\alpha}{n+\alpha-1}}\\
	&\geq(n+\alpha)\mathcal{V}_\alpha^{\frac{1}{n+\alpha}}
	\Big(\frac{1+b_0}{e^{2r_0b_1+b_0}}\Big)^{\frac{n+\alpha-1}{n+\alpha}}
	\Big(\int_\Omega wf^{\frac{n+\alpha}{n+\alpha-1}}\Big)^{\frac{n+\alpha-1}{n+\alpha}}.
\end{aligned}
$$
\end{proof}

\section*{Acknowledgements}

We would like to thank referees for their helpful comments and valuable suggestions.

\bibliographystyle{plain}

\bibliography{BEsobo}
\newpage

\noindent\mbox{Yuxin Dong and Lingen Lu} \\
\mbox{School of Mathematical Sciences}\\
\mbox{220 Handan Road, Yangpu District}\\
\mbox{Fudan University}\\
\mbox{Shanghai, 200433}\\
\mbox{P.R. China}\\
\mbox{\textcolor{blue}{yxdong@fudan.edu.cn}}\\
\mbox{\textcolor{blue}{lulingen@fudan.edu.cn}}

%\newpage

~

\noindent\mbox{Hezi Lin}\\
\mbox{School of Mathematics and Statistics \& Key Laboratory of Analytical Mathematics}\\
\mbox{and Applications (Ministry of Education) \& FJKLAMA}\\
\mbox{Fujian Normal University}\\
\mbox{Fuzhou,  350108}\\
\mbox{P.R. China}\\
\mbox{\textcolor{blue}{ lhz1@fjnu.edu.cn}}

\end{document}